\documentclass[a4paper,12pt]{article}
\usepackage[margin=1in]{geometry}  

\usepackage{graphicx}              
\usepackage{amsmath}
\usepackage{amscd}               
\usepackage{amsfonts} 
\usepackage{amssymb}             
\usepackage{amsthm}                
\usepackage{enumerate}

\newtheorem{theorem}{Theorem}[section]

\newtheorem{lemma}[theorem]{Lemma}

\theoremstyle{definition}

\DeclareMathOperator{\id}{id}

\newcommand{\cL}{\mathcal{L}}
\newcommand{\cO}{\mathcal{O}}

\newcommand{\Z}{\mathbb{Z}}      
\newcommand{\N}{\mathbb{N}}

\newcommand{\PP}{\mathbb{P}}
\newcommand{\TV}{\text{TV}}
\newcommand{\Tr}{\text{Tr}}
\begin{document}
\title{Mixing time of generalized Rosenthal walk on $SO(n)$}
\author{Yunjiang Jiang}

\maketitle
\begin{abstract}We prove that a uniformized variant of both the Rosenthal walk \cite{Rosenthal} and the Kac random walk \cite{Kac} on $SO(n)$ mixes in $\cO(n^3)$ steps in total variation distance. The proof also extends easily to Rosenthal walk with fixed angle $\theta \neq \pi$. To the best of our knowledge, this is the first polynomial time bound for both walks. The techniques employed are mainly from representation theory of $SO(n)$. But a crucial new ingredient is the interpretation of the Fourier coefficients of the character ratio as counting the number of particle cascade paths arising from the classical branching rules.
\end{abstract}

Consider the random walk on $SO(N)$ defined by $X_0 = \id$, $X_j = X_{j-1} S_j$, where $S_j$ are iid random elements in $SO(N)$ uniformly supported on the conjugacy class of $R(1,2;\theta)$. In other words, each step we take a random rotation of angle $\theta$ along a uniformly chosen two-plane. The goal is to bound the total variation mixing time to its stationary distribution, namely the Haar measure.

For two probability measures on the same space $\Omega$, recall the following consequence of the Cauchy-Schwarz inequality:
\begin{align*}
 \| \mu - \nu\|_{\TV}^2 &= \| \frac{d\mu}{d\nu} -1\|_{\cL^1(d\nu)}^2 \\
&\le  \| \frac{d\mu}{d\nu} -1\|_{\cL^2(d\nu)}^2. 
\end{align*}
Here we are assuming the density $\frac{d\mu}{d\nu}$ exists, but even when it doesn't, the first equality still holds if we restrict the integral to where $\frac{d\mu}{d\nu} < \infty$. However, the $\cL^2$ norm might be infinite. 

Now we specialize to $\Omega = SO(N)$ and $\nu = U$, the Haar measure. Then by Plancherel's theorem, the right hand side equals (see \cite{PD} Chapter 3)
\begin{align*}
\sum_{\rho \in \widehat{SO(N)}} \Tr \widehat{(\mu - U)}(\rho)\widehat{(\mu - U)}(\rho)^T =  \sum_{\rho \neq 1} \Tr \widehat{\mu}(\rho)\widehat{\mu}(\rho)^T.
\end{align*}
To see the last equality, simply observe that for $\rho=1$, and any probability measure $\pi$, $\widehat{\pi}(\rho) = 1$, hence $\widehat{(\mu - U)}(\rho) = 0$; for any other $\rho$, $\widehat{U}(\rho) = 0$ since $\rho(g) \widehat{U}(\rho) = \widehat{U}(\rho)$ for any $g \in SO(N)$, by left-invariance of $U$.

Furthermore, we have the convolution identity $\widehat{\mu * \nu}(\rho) = \widehat{\mu}(\rho) \widehat{\nu}(\rho)$. Now let $S_j$'s be distributed according to $\mu$. Then since $\mu$ is constant on conjugacy classes, by Schur's lemma, we know that $\widehat{\mu}(\rho)$ is always a constant times the identity matrix. Therefore
\begin{align*}
 \Tr \widehat{\mu^{*t}}(\rho) \widehat{\mu^{*t}}(\rho)^T = d_\rho^2 (c_\rho / d_\rho)^{2t},
\end{align*}
where $c_\rho = \Tr \widehat{\mu}(\rho)$. 

Now we further restrict to the case of $N=2n+1$. Then the irreducible representations of $SO(N)$ are indexed by $n$-tuples $(0 \le a_1 \le a_2 \le \ldots \le a_n)$ \cite{Goodman}. We will denote $d_a$ and $c_a = c_a(\theta)$ the dimension and the character value evaluated at $R(1,2;\theta)$ of the irreducible representation labeled by the $n$-tuple $a$. Note that Rosenthal \cite{Rosenthal}, \cite{Adams} used the following slightly different labeling convention $(a_1 + 1/2, a_2 + 3/2, \ldots, a_n + n-1/2)$.  He gave the following explicit formulae for $d_a$ and $c_a$ by taking suitable limits in the Weyl character formula:
\begin{align*}
 d_a &= \frac{2^n}{1! 3! \ldots (2n-1)!} \prod_{q=1}^n (a_q + q -1/2) \prod_{1 \le s < r \le n} [(a_r+r-1/2)^2 - (a_s + s -1/2)^2]\\
 c_a(\pi) &= \frac{(2n-1)!}{(2 \sin (\theta/2))^{2n-1}} \sum_{j=1}^n \frac{\sin (a_j + j -1/2) \theta}{(a_j + j -1/2) \prod_{r \neq j} (a_r + r -1/2)^2 - (a_j + j -1/2)^2}
\end{align*}

By restricting $\rho_a$ to the caononical copy of $SO(2)$ in $SO(N)$, it decomposes into a nonnegative linear combination of irreducible representations of $SO(2)$, which are of the form $t \mapsto t^k$ for integer $k$, $|t| =1$. When $k=0$, we have the trivial representation. Now since $R(1,2;\theta)$ is an element of $SO(2) \subset SO(N)$, we can expression $\rho_a(R(1,2;\theta))$ in terms of irreducible representations of $SO(2)$, so for $\alpha_j \ge 0$,
\begin{align*}
\rho_a(R(1,2;\theta)) = \sum_{j=0}^m \alpha_j \cos(j\theta).
\end{align*}

Clearly, $c_a(\pi) = \sum_j\alpha_j (-1)^j \ge \alpha_0 -\alpha_1$. The $\alpha_j$'s admit the following interpretation. By Weyl's character formula, one can express the restriction of $\rho_a$ from $SO(N)$ to $SO(N-1)$ as follows:
\begin{align*}
 \rho_a^{(2n+1)}|_{SO(2n)} &= \bigoplus_{|b_1| \le a_1 \le b_2 \le \ldots \le b_n \le a_n} \rho_b^{(2n)}\\
\rho_b^{(2n)}|_{SO(2n-1)} &= \bigoplus_{|b_1| \le a_1 \le b_2 \le \ldots \le b_n } \rho_a^{(2n-1)}
\end{align*}
so each irreducible representation of $SO(N-1)$ appears at most once when restricted from $SO(N)$. This yields a combinatorial description of $\alpha_j$, namely the number of shrinking ensemble paths with the particle non-intersecting condition as given above such that the topmost particle arrives at $\pm j$ in $N-2$ steps (since we are going from $SO(N)$ to $SO(2)$). We have the following two immediate consequences:
\begin{lemma}
 $\alpha_1 \ge \alpha_2 \ge \ldots$, so the $\alpha_j$ are monotone non-increasing starting at $j=1$. 
\end{lemma}
\begin{proof}
 We will let $\beta_j$ be the number of ensemble paths that leads to the $a^{(3)}_1 = j$ at the $SO(3)$ level. Then condition on $a^{(3)}_1 = j$, $j \ge 1$, the next step one could get $b^{(2)}_1 \in [-j,j]$, each with multiplicity $1$. Thus if $1 \le r < s \in \Z$, the conditional multiplicity of $a^{(2)} = \pm r$ is no less than that of $a^{(2)}=\pm s$. This shows $\alpha_r \ge \alpha_s$. Since there is only one copy of $a^{(2)} = 0$, $\alpha_0$ is not necessarily $\ge \alpha_1$.  
\end{proof}
\begin{lemma}
 Let $\beta_j$ be as defined in the previous proof. Then $2\beta_j = \alpha_j - \alpha_{j+1}$ for $j \ge 1$, and $\beta_0 = \alpha_0 - \frac{1}{2} \alpha_1$.  Furthermore we have $ \beta_j \ge \frac{2j+1}{2j+3} \beta_{j+1}$.
\end{lemma}
\begin{proof}
 The relation between $\beta$ and $\alpha$ follows from the following two equations:
\begin{align*}
 \alpha_0 &= \sum_{i=0}^m \beta_i\\
 \alpha_j &= 2\sum_{i=j}^m \beta_j.
\end{align*}
These are consequences of the interlacing condition for the branching rule from $SO(3)$ to $SO(2)$.
Next we condition on $a_1^{(5)} \le b_2^{(4)} \le a_2^{(5)}$. Then $b_1^{(4)}$ can be any integer in the range $[-a_1^{(5)},a_1^{(5)}]$, each with multiplicity $1$. So the number of branching paths down to $SO(3)$ with $a_1^{(3)}=j$ is given by $\beta_j':=[1 + 2 (j \wedge a_1^{(5)})] I_{\{j \le b_2^{(4)}\}}$, where $\beta_j':= \beta_j^{a_1^{(5)},b_2^{(4)},a_2^{(5)}}$ is the conditional $\beta$ so to speak. Then it is clear by case analysis that
\begin{align*}
 \beta_j' \ge \frac{1+2j}{3+2j} \beta_{j+1}'.
\end{align*}
Now summing over all the conditions $a_1^{(5)} \le b_2^{(4)} \le a_2^{(5)}$ we obtain the second assertion.
\end{proof}
Observe now that $\sum_{j\ge 0} \alpha_j = d_a$, since each shrinking ensemble path represents one copy of some irreducible representation of $SO(2)$, which are all one-dimensional. We shall use the notation $\tilde{\alpha}_j = \alpha_j / d_a$, and similarly $\tilde{\beta}_j = \beta_j / d_a$. Note that $\sum_{j \ge 0} \tilde{\beta}_j \neq 1$. 
  
The last lemma implies the following bound:
\begin{lemma}
If $\tilde{\alpha}_0 > \tilde{\alpha}_1$, then
 $\tilde{\alpha}_0 \le \cO((\tilde{\alpha}_0 - \tilde{\alpha}_1)^{1/3}) \le \cO(\sum_{j\ge 0} \tilde{\alpha}_j (-1)^j)$. 
\end{lemma}
Unfortunately, the condition in the above lemma only holds for small representations.
\begin{proof}
 Since $\alpha_j - \alpha_{j+1} = 2\beta_j \le 2(1+2j) \beta_0$ by telescoping, we have 
\begin{align*}
\alpha_0 - \alpha_{j+1} =  \alpha_0 - \alpha_1 + \alpha_1 - \alpha_2 + \ldots + \alpha_j - \alpha_{j+1} \le 3\sum_{i=0}^j (1+2j) \beta_0,
\end{align*}
i.e., $\alpha_{j+1} \ge \alpha_0 - (2j+3)j \beta_0$. Now using the fact that $\sum_{j \ge 0} \tilde{\alpha}_j = 1$, we have for all $k \ge 0$,
\begin{align*}
 (k+1) \tilde{\alpha}_0 - \sum_{j=1}^k (2j+3)j \tilde{\beta}_0 = (k+1) \tilde{\alpha}_0 - f(k) \tilde{\beta}_0 \le 1,
\end{align*}
where $f(k)$ has cubic growth. So 
\begin{align*}
 \tilde{\alpha}_0 \le \frac{1}{k+1} + \frac{f(k)}{k+1} \tilde{\beta}_0.
\end{align*}
Optimizing over $k$ on the right hand side, we obtain 
\begin{align*}
\tilde{\alpha}_0 &\le c (\tilde{\beta}_0)^{1/3} \\
&= c (\tilde{\alpha}_0 - \frac{1}{2} \tilde{\alpha}_1)^{1/3}\\
&\le c (\tilde{\alpha}_0 - \tilde{\alpha}_1)^{1/3}\\
&\le c(\sum_{j \ge 1} \tilde{\alpha}_j (-1)^j)^{1/3}.
\end{align*}
Where in the last inequality we used monotonicity of $\alpha_j$, $j \ge 1$.
\end{proof}

Next we come back to bounding the total variation distance using the character sum:
\begin{align*}
 \| \mu^{*t} - U\|_{\TV} \le \sum_{a \neq (0^n)} d_a^2 (\frac{c_a}{d_a})^{2t}, 
\end{align*}
where $(0^n)$ is the label for the trivial representation. We first give an upper bound on $d_a$. Since we know $d_{(0^n)} =1$, it suffices to look at the ratio:
\begin{align*}
 \frac{d_a}{d_{(0^n)}} &= \prod_{q=1}^n \frac{a_q + q -1/2}{q-1/2} \prod_{r > s} \frac{(a_r + r -1/2)^2 - (a_s + s -1/2)^2}{(r-1/2)^2 - (s-1/2)^2}
\end{align*}
Now suppose $a = (0^{n-k}, a_{n-k+1}, \ldots, a_{n})$, i.e., the first $n-k$ coordinates are all $0$. Then we have the following bound:
\begin{align*}
 \frac{d_a}{d_{(0^n)}} &= \prod_{j=n-k+1}^n \frac{a_j + j - 1/2}{j-1/2} \prod_{s=1}^{n-k} \prod_{r=n-k+1}^n \frac{a_r + r -s}{r-s} \frac{a_r + r+ s -1}{r+s -1}\\
&\le c\prod_{r=n-k+1}^n \exp(a_r \sum_{s=1}^{n-k} \frac{1}{r-s}) \exp(a_r \sum_{s=1}^{n-k} \frac{1}{r+s})\\
&\le c\prod_{r=n-k+1}^n (\frac{r-1}{r-n+k})^{a_r} (\frac{r+n-k}{r+1})^{a_r}\\
&\le c\frac{n^{\sum_{r=n-k+1}^n a_r}}{\prod_{j=1}^k j^{a_{n-k+j}}}
\end{align*}

\section{Upper bound on the mixing time of the uniform Rosenthal walk}
Consider the following random walk on $SO(n)$. Each step, an independent uniform angle $\theta \in [0,2\pi)$ is chosen and then an independent and uniform element in the conjugacy class of $R(1,2;\theta)$ is chosen. The walk then multiplies itself by this doubly stochastic element to reach the next step. This is closely related to the Kac random walk on $SO(n)$, where instead of choosing a random element in the conjugacy class, one chooses a random pair of coordinates $1 \le i < j \le n$, and multiplies the chain by the random matrix $R(i,j;\theta)$. In the Kac case, no polynomial mixing time upper bound is known \cite{YJ} (though see \cite{Oliv}, \cite{CJL}, and \cite{CCRLV} for results in other modes of convergence). Here we present a polynomial mixing time argument for the corresponding Rosenthal walk. 

  First we need a simple lemma
\begin{lemma}
 Consider a random walk $K$ generated by $\mu$ on a compact group $G$ which is constant on conjugacy classes and another one $J$ generated by $\nu$ which can be of any type (doesn't even need to be ergodic). Then the $\cL^2$ distance to stationarity satisfies the following censoring inequality:
\begin{align*}
 \|\mu^{t_1} * \nu * \mu^{t_2} * \ldots * \nu *\mu^{t_k} - U \|_2 \le \|K^t - U\|_2. 
\end{align*}
where $\sum_i t_i = t$ and $t_i \ge 0$.
\end{lemma}
\begin{proof}
 This simply follows from the Plancherel's identity. It clearly suffices to show that for each irreducible representation $\rho$ of $G$,
\begin{align*}
\Tr \widehat{\mu^{t_1} * \nu *\mu^{t_2}}(\rho) \le \Tr \widehat{\mu}^t(\rho).
\end{align*}
 But this is true simply becuase $\hat{\mu}(\rho) = c I$ for some constant $c$ so can be collected in the front of the final matrix product.
\end{proof}

Now we will consider two auxilliary random walks to the uniform Rosenthal walk. Let $K$ be the Rosenthal walk with $\theta$ uniformly supported on $[\epsilon,2\pi -\epsilon]$, and $J$ be the unconditional walk. if we condition the Rosenthal walk on $\theta$ in the range $[\epsilon,2\pi -\epsilon]$, then we obtain $K$ above. So the strategy now will be to show that $K$ mixes in $O(n^2 \log n)$ steps and that with high probability in $O(n^2 \log n)$ steps there will be that many steps where $K$ is taken.

The second claim is easy to deal with. Let $\#K$ denote the number of steps in which $\theta \in [\epsilon, 2\pi - \epsilon]$. 
\begin{align*}
 \PP_{c t} [ \# K \ge t] = \sum_{j=t}^{ct}\binom{c t}{t} (1-\epsilon)^j \epsilon^{ct -j}.
\end{align*}
For $c$ sufficiently large, the distribution of $\#K$ converges to a Gaussian centered at $ct (1-2\epsilon)$ with standard deviation about $\sqrt{ct}$. So if we choose $c > 2$ then certainly $\#K \ge t$ with overwhelming probability.

So now it suffices to bound the $\cL^2$ mixing time of $K$. 
\begin{itemize}
 \item Regime 1. When $r(\pi) = \sum_{j=0}^\infty \alpha_j (-1)^j > 1/6$, we can infer that $\alpha_0 > \alpha_1$, so the full monotonicity of the Fourier coefficients holds. Let 
\begin{align*}
T_a(j):= \frac{(2n-1)!}{2^{2n-1}} \frac{(-1)^j}{(a_j + j -1/2) \prod_{r \neq j} ((a_r+r-1/2)^2 - (a_j + j -1/2)^2)}. 
\end{align*}
 From the Rosenthal bound on $r(\pi)$, 
\begin{align*}
 r_a(\pi) \le \sum_{j=1}^n T_a(j),
\end{align*}

we see that if $a_n > 12 n$, then $r_a(\pi) < 1/6$. This is because the probability measure on $[n]$ defined by 
\begin{align*}
\mu(j):=T_{0^n}(j) = \frac{(2n-1)!}{2^{2n-1}} \frac{(-1)^j}{(j-1/2)\prod_{r \neq j}[(r-1/2)^2-(j-1/2)^2]}, 
\end{align*}
is concentrated near $j=1$ with a Gaussian decay of window size $n^{1/2}$, as easily confirmed by Stirling approximation. Furthermore for any $a \in (\N^n)^{\uparrow}$, $T_a(j) \le \mu(j)$ using the monotonicity $a_0 \le a_1 \le \ldots \le a_n$.  Therefore $\sum_{j > cn^{1/2}} T_a(j) = o(1)$, and we only need to focus on $j < cn^{1/2}$. For the latter we simply compare $T_a(j)$ with $\mu(j)$ and bound the ratio uniformly:

\begin{align*}
 \frac{(a_j + j -1/2)^2 - (a_n + n -1/2)^2}{(j-1/2)^2 - (n-1/2)^2} \ge \frac{a_j + a_n + j + n -1}{j+n-1}.
\end{align*}
So if $a_n > 12 n$, this ratio is $> 6$. In fact this is true for all $j \in [n]$. Therefore by monotonicity of expectation, we easily get $r_a(\pi) < 1/6$. 

Now observe that $a_n$ is precisely the number of nonzero $\alpha_j$'s in the Fourier expansion of $r(\theta)$, using the branching rule characterization. Therefore by monotonicity of $\alpha_j$'s, 
\begin{align*}
\alpha_1 &\ge \frac{\sum_{j=1}^{a_n} \alpha_j}{a_n}\\
 &\ge \frac{\alpha_1 + \alpha_3 + \ldots }{a_n}\\
 &\ge \frac{1}{2} \frac{1-r(\pi)}{a_n}. 
\end{align*}
Finally we have for $\cos \theta < 1 - \epsilon$ (which is the case with our assumption that $\theta$ is bounded away from $0$ mod $2\pi$), that
\begin{align*}
 r(\theta) < 1 - (1-\cos \theta) \alpha_1 < 1 - (1-\cos \theta)\frac{1-r(\pi)}{2a_n}.
\end{align*}
Now if we let $A$ be all the irreducible representation indices of $SO(2n+1)$ with the property that $r_a(\pi) > 1/6$, then Rosenthal clearly implies that
\begin{align*}
 \sum_{a \in A} d_a r_a(\pi)^{n \log n} = o(1),
\end{align*}
 which in turn gives
\begin{align*}
\sum_{a \in A} d_a r_a(\theta)^{\frac{2 a_n}{1 - \cos \theta} n \log n} = o(1).
\end{align*}

\item Regime 2. From Rosenthal's work \cite{Rosenthal}, we know the character ratio at a particular $\theta$ is bounded as follows:
\begin{align*}
 r(\theta) \le \frac{(2n-1)!}{(2\sin \theta/2)^{2n-1}} \sum_{j=1}^n \frac{1}{a_j |\prod_{r \neq j} (a_r + r)^2-(a_j +j)^2|}.
\end{align*}

Let us denote $W_s = \prod_{r > s} |\frac{(a_r + r)^2 -(a_s + s)^2}{r^2 - s^2}|$. Using the fact that the character ratio is 1 for the trivial representation, taking the smallest $W_0$ gives a bound of $r(\theta)$ as follows:
\begin{align*}
 r(\theta) \le \frac{1}{(\sin \theta/2)^{2n-1}} \sum_{j=1}^n \frac{1}{W_0}.
\end{align*}

 We also have the control over the gap between the smallest and largest $W_s$, $\max_s W_s \le (\min_ s W_s)^{2n}$. This follows from the simple and crude observation that for any fixed $j$, 
\begin{align*}
 (a_n + a_j + n + j) &\ge \frac{1}{2} (a_r + a_s + r + s)\\
 (a_n + a_j + n + j) &\ge (a_r - a_s + r - s).
\end{align*}
Actually the bound above is the best one could hope for, as it is achieved by the following family of representations $(0^{n-1},k)$, as $k \to \infty$. 

But now we have $r(\theta) \le \frac{1}{(\sin \theta/2)^{2n-1}} (\min_j W_j)^{-1} n$,  whereas the dimension $d \le (\max_j W_j)^{2n/2}$ by throwing away the corresponding terms for the trivial representation, and estimating the factor $\prod_q (a_q + q -1/2)$ trivially by the product of the $W_j$'s. So denoting by $W_0$ the minimizer of $W_j$, we need $t$ such that the following expression is small:
\begin{align*}
[W_0^{2n}]^{n} [\sin^{-(2n-1)}[ \theta/2] W_0^{-1}n]^t 
 \end{align*}

It is important that our upper bound for $r(\theta)$ is $< 1$. to accomplish that we need $W_0 \ge e^{\Omega(n)}$. In that event, we certainly can take $t = \Omega(n^2)$ and get the required bound.

\item Regime 3. Now we need to consider the case when our upper bound for $r(\theta) \ge 1$. One possibility is taken care of by Regime 1. So it suffices to consider the situation where $r(\pi) < 1/6$. But that means $\alpha_0 < \frac{7}{12}$ by an easy computation using the fact that $\alpha_0 - \alpha_1 + \alpha_2 \le r(\pi)$. Integrating over $\theta \in [\epsilon, 2\pi -\epsilon]$, we obtain
\begin{align*}
 \frac{1}{2\pi - 2\epsilon}\int_\epsilon^{2\pi -\epsilon} r(\theta) d\theta = \alpha_0 - \sum_{j=1}^\infty \frac{1}{j (\pi - \epsilon)} \sin (\epsilon j).
\end{align*}
It is clear that $\frac{\sin \epsilon j}{j(\pi - \epsilon)} \le \frac{1}{\pi -\epsilon}$.

 So for fixed $\epsilon \in (0,\pi -1)$, we have the uniform bound $\int_{[\epsilon,2\pi-\epsilon]}r(\theta) d\theta \le 1 - c_\epsilon(1-\alpha_0)$. So in other words, $r(\theta) < 1 - \delta$ for some $\delta > 0$. On the other hand, the dimension is bounded by
\begin{align*}
 d \le (W_0^{2n})^n \le \exp(\cO(n^3)).
\end{align*}
So taking $t = \Omega(n^3)$ suffices to kill all the harmonics in this regime.
\end{itemize}
 
In summary, the mixing time of the uniform Rosenthal walk is $\cO(n^3)$. There is a corresponding lower bound in $\cL^2$ of order $n^2$ (see Porod and Rosenthal). The above argument clearly shows that the mixing time of the uniform walk is of the same order as for a fixed $\theta$ (or uniform $\theta$ is a region bounded away from $0$). So the conjectured total variation mixing time is still $n \log n$.

\section{Acknowledgement}
I received generous help from Bob Hough in making the connection of Fourier expansion of the character ratio with restriction of irreducible representations as well as some preliminary bounds for the character ratio. Thanks are due to Daniel Bump \cite{Bump} for explaining to me how the classical branching rule for $SO(n)$ works. I would also like to thank Persi Diaconis and Arun Ram for help with literature search.

\bibliography{NSF_postdoc}{}

\begin{thebibliography}{10}

\bibitem{Adams}
J.~Frank Adams.
\newblock {\em Lectures on {L}ie groups}.
\newblock W. A. Benjamin, Inc., New York-Amsterdam, 1969.

\bibitem{Bump}
Daniel Bump.
\newblock {\em Lie groups}, volume 225 of {\em Graduate Texts in Mathematics}.
\newblock Springer-Verlag, New York, 2004.

\bibitem{CCRLV}
Eric~A. Carlen, Maria~C. Carvalho, Jonathan Le~Roux, Michael Loss, and
  C{\'e}dric Villani.
\newblock Entropy and chaos in the {K}ac model.
\newblock {\em Kinet. Relat. Models}, 3(1):85--122, 2010.

\bibitem{CJL}
Eric~A. Carlen, Jeffrey~S. Geronimo, and Michael Loss.
\newblock Determination of the spectral gap in the {K}ac model for physical
  momentum and energy-conserving collisions.
\newblock {\em SIAM J. Math. Anal.}, 40(1):327--364, 2008.

\bibitem{PD}
Persi Diaconis.
\newblock {\em Group representations in probability and statistics}.
\newblock Institute of Mathematical Statistics Lecture Notes---Monograph
  Series, 11. Institute of Mathematical Statistics, Hayward, CA, 1988.

\bibitem{Kac}
Persi Diaconis and Laurent Saloff-Coste.
\newblock Bounds for {K}ac's master equation.
\newblock {\em Comm. Math. Phys.}, 209(3):729--755, 2000.

\bibitem{Goodman}
Roe Goodman and Nolan~R. Wallach.
\newblock {\em Representations and invariants of the classical groups},
  volume~68 of {\em Encyclopedia of Mathematics and its Applications}.
\newblock Cambridge University Press, Cambridge, 1998.

\bibitem{YJ}
Yunjiang Jiang.
\newblock Total variation bounds of kac random walk.
\newblock {\em Ann. Appl. Probab.}, 2011.

\bibitem{Oliv}
Roberto~Imbuzeiro Oliveira.
\newblock On the convergence to equilibrium of {K}ac's random walk on matrices.
\newblock {\em Ann. Appl. Probab.}, 19(3):1200--1231, 2009.

\bibitem{Rosenthal}
Jeffrey~S. Rosenthal.
\newblock Random rotations: characters and random walks on {${\rm SO}(N)$}.
\newblock {\em Ann. Probab.}, 22(1):398--423, 1994.

\end{thebibliography}
\bibliographystyle{plain}

\end{document}